\documentclass[letterpaper, 10 pt, conference]{ieeeconf}
\IEEEoverridecommandlockouts
\usepackage[table,usenames,dvipsnames]{xcolor} 
\usepackage[noadjust]{cite}
\usepackage{amsmath,amssymb,amsfonts,amsthm,dsfont,mathtools}
\usepackage{adjustbox}
\usepackage{graphicx,tabularx,adjustbox,color}
\usepackage{multirow}
\usepackage[font=footnotesize]{caption}
\usepackage[font=footnotesize]{subcaption}
\usepackage{algorithmic}
\usepackage[ruled,vlined]{algorithm2e}
\usepackage{stackengine}

\usepackage{dirtytalk}
\allowdisplaybreaks
\renewcommand{\baselinestretch}{.985}

\usepackage[breaklinks=true, colorlinks, bookmarks=true, citecolor=Black, urlcolor=Violet,linkcolor=Black]{hyperref}

\newtheoremstyle{iremark}
  {\topsep}   
  {\topsep}   
  {\upshape}  
  {0pt}       
  {\itshape}  
  {:}         
  {5pt plus 1pt minus 1pt} 
  {\thmname{#1}\thmnumber{ \itshape#2}\thmnote{ (#3)}} 

\theoremstyle{iremark}
\newtheorem{proposition}{\it{Proposition}}

\newtheorem{theorem}{\it{Theorem}}
\newtheorem{definition}{\it{Definition}}

\newtheorem*{remark*}{Remark}
\newtheorem{remark}{\it{Remark}}

\usepackage{xpatch}
\makeatletter
\AtBeginDocument{\xpatchcmd{\@thm}{\thm@headpunct{.}}{\thm@headpunct{:}}{}{}}
\makeatother

\usepackage[outdir=./]{epstopdf}

\newcommand{\ba}{\begin{align}}
\newcommand{\ea}{\end{align}}
\newcommand{\fr}{\frac}

\title{\LARGE \bf Periodic Event-Triggered Boundary Control of Neuron Growth with Actuation at Soma} 

\author{
Cenk Demir$^1$, 
Mamadou Diagne$^1$ 
and Miroslav Krstic$^1$ \thanks{$^1$Department of Mechanical and Aerospace Engineering, UC San Diego, 9500 Gilman Drive, La Jolla, CA, 92093-0411, {\tt\small cdemir@ucsd,edu, mdiagne@ucsd.edu, krstic@ucsd.edu}}}

\begin{document}
\maketitle
\begin{abstract}

Exploring novel strategies for the regulation of axon growth, we introduce a periodic event-triggered control (PETC)  to enhance the practical implementation of the associated PDE backstepping control law. Neurological injuries may impair neuronal function, but therapies like Chondroitinase ABC (ChABC) have shown promise in improving axon elongation by influencing the extracellular matrix. This matrix, composed of extracellular macromolecules and minerals, regulates tubulin protein concentration, potentially aiding in neuronal recovery. The concentration and spatial distribution of tubulin influence axon elongation dynamics. Recent research explores feedback control strategies for this model, leading to the development of an event-triggering control (CETC) approach. In this approach, the control law updates when the monitored triggering condition is met, reducing actuation resource consumption. Through the meticulous redesign of the triggering mechanism,   we introduce a periodic event-triggering control (PETC), updating control inputs at specific intervals, but evaluating the event-trigger only periodically, an ideal tool for standard time-sliced actuators like ChABC. PETC is a step forward to the design of practically feasible feedback laws for the neuron growth process.  The PETC strategy establishes an upper bound on event triggers between periodic examinations, ensuring convergence and preventing Zeno behavior. Through Lyapunov analysis   , we demonstrate the local exponential convergence of the system with the periodic event-triggering mechanism in the $L^2$-norm sense. Numerical examples are presented to confirm the theoretical findings.
\end{abstract}

\section{Introduction}
\label{sec:intro}

Neurons, as fundamental components of neural networks, are intricately involved in sensory processing \cite{squire2012fundamental}. They play a critical role in acquiring and interpreting sensory information by transmitting electrical signals through their specialized structures. This communication occurs between neurons via their axons, which function as cellular wires. Axons consist of groups of proteins called ``tubulin." The dynamics of these proteins facilitate the elongation of axons, allowing them to reach the target neuron, establish synaptic connections, and complete the transmission process. However, neurological diseases or injuries can disrupt or completely halt this transmission such as Alzheimer's disease \cite{maccioni2001molecular} and spinal cord injuries \cite{liu1997neuronal}. In such complications, neurons may degenerate, leading to axonal shrinkage or an inability to reach target neurons for signal transmission. Until recently, the prevailing belief was that injured neurons could not regenerate to complete the transmission process \cite{huebner2009axon}. However, recent research has clarified that regeneration is possible under certain conditions. A recently developed therapy method, Chondroitinase ABC, shows promising potential for regenerating neurons \cite{bradbury2011manipulating,karimi2010synergistic}. With this therapy, it is possible to stimulate neuron elongation.

Advancements in ChABC and other therapy methods could help neuron regeneration, but the key factor of this regeneration will be the degree of neuronal elongation which is a process regulated by tubulin dynamics. This dynamical process is modeled with different mathematical approaches and assumptions \cite{van1994neuritic,garcia2017continuum}. However, a more comprehensive model for tubulin-based axon elongation is proposed in \cite{mclean2004continuum}. This model comprises a diffusion-reaction-advection partial differential equation (PDE) governing tubulin evolution along the axon, alongside an ordinary differential equation describing tubulin evolution at the axon's far end, namely the growth cone, and the axon's length over time. This specific model takes the form of a Stefan-type PDE, a well-studied model in the literature \cite{gupta2017classical}. Furthermore, mathematical stability analysis and the effect of model parameters of the axon elongation are detailed in \cite{mclean2006stability}.

In recent times, PDE systems have garnered significant attention across various disciplines. Control system engineering, in particular, has emerged as a prominent field offering control strategies for these systems. A pioneering research direction in this domain is boundary control of PDE systems which is based on the work \cite{krstic2008boundary}. 
Following these important initial studies, researchers have broadened their focus to include backstepping-based boundary control of various types of PDEs, systems that combine PDEs with ODEs, as well as systems that involve multiple interacting PDEs \cite{krstic09,susto2010control, tang2011state}. While previous studies focus on constant domain size in time, there are significant works on the global results of moving domains in time \cite{COCV_2003__9__275_0,Petit10,maidi2014,Zhelin20,ecklebe2021toward,izadi2015pde, diagne2015feedback}.  This line of work has been extended to derive local stability results for moving boundary nonlinear hyperbolic PDEs \cite{ bastin2019boundary, buisson2018control, yu2020bilateral}. Achieving stability for nonlinear parabolic PDEs with moving boundaries, especially without using the maximum principle, has been challenging until our recent study on axonal growth \cite{demir2021neuroncontrol,demir2022neuron}.

While the control methods discussed above operate continuously, certain technologies necessitate interventions only when needed, driven by limitations in energy, communication, and computation \cite{heemels2012introduction}. These constraints suggest  control actions that are executed only when needed and thereby enhance actuation resource usage, which is called the event-triggering control. This concept was initially developed for linear systems in \cite{aaarzen1999simple,heemels2008analysis} and nonlinear systems in \cite{kofman2006level}. Following these establishments, it was proposed for infinite dimensional systems as event-triggered boundary control of PDE systems in \cite{espitia2016event}. Static and dynamic triggering mechanisms have been developed for ODEs and various classes of PDEs, specifically, for Stefan problem in \cite{rathnayake2022event2,rathnayake2022event}. Periodic event-triggered control (PETC) and self-triggered control are proposed in    \cite{rathnayake2023periodic} and  \cite{rathnayake2023self}, respectively, for a class of  reaction diffusion PDEs. Moreover, this method finds application in ensuring the safety and convergence to the set point of the Stefan problems in the presence of actuator dynamics in \cite{koga2023event}.

In  addressing the neuron growth problem, a dynamic event-triggering mechanism for a coupled PDE-nonlinear ODE with a moving boundary is introduced in \cite{demir2023event}. Motivated  by the  feasibility challenges encountered in real-time implementation when applying feedback control to biological systems for medical treatements, we convert the continuous time event-triggered control (CETC) design by supplementing a periodic sampling rule. This introduces a dynamic periodic event-triggering control (PETC) approach, where the triggering function is only checked periodically while the  control input being updated aperiodically. The PETC improves the practical implementation of the control law because  it can be applied to standard time-sliced actuators (like ChABC) for axon growth. The strategy involves deriving a novel triggering condition and establishing an upper bound on the continuous-time event trigger between two periodic examinations, explicitly derived as a sampling period in our study. The problem statement in our current study diverges from that addressed in \cite{rathnayake2023observer}, where observer-based PETC was utilized for the classical one-phase Stefan problem. While the outcome in \cite{rathnayake2023observer} pertains to a system exhibiting \emph{geometric nonlinearity,} the neuron growth process involves both \emph{geometrical and analytical nonlinearities} leading to a \emph{local} convergence result.  The PETC boundary controller  guarantees $L^2$ local exponential convergence in closed-loop, akin CETC.

The paper's structure is as follows: Section II introduces the tubulin-driven axon growth model, including the steady-state solution and reference error analysis, and control law design. Section III presents the event-triggering mechanism and results from prior work. Section IV introduces a novel periodic event-triggering control for the axon growth problem. Finally, Section V presents numerical simulation results.

\vspace{-0.5em}

\section{Tubulin-Driven Axon Growth Modeling and Control} \label{sec:model} 

\vspace{-0.25em}

This section introduces the tubulin-driven axonal growth model, a coupled moving boundary PDE with ODEs, and a locally exponentially stabilizing continuous-time boundary control law. 

\subsection{Understanding axon growth}

\subsubsection{A model with a moving boundary PDE}
Tubulin, a collection of proteins, facilitates the development of a newly formed axon. Under
the assumption that unattached tubulin molecules along the
axon is insignificant, and considering that only tubulin
molecules are accountable for axon growth, the newborn
evolution in time and space can be modeled as follows \cite{diehl2014one, diehl2016efficient}.
\begin{align}\label{sys1} 
c_t(x,t) = D &c_{xx} (x,t) - a c_x (x,t) - g c(x,t),  \\
\label{sys2} c_x(0,t)+ c(0,t) &=  - q_{\rm s}(t), \\
\label{sys3} c(l(t),t) = c_{\rm c}& (t), \\
\label{sys4}  \dot{c}_{\rm c}(t) =  \tilde{a} &c_{\rm c}(t) - \beta c_x(l(t), t) -\kappa c_{\rm c}^2(t)+c_{\infty}\tilde{r}_{\rm g}, \\
\label{sys5} \dot{l}(t) = r_{\rm g} &(c_c(t)-c_{\infty})
\end{align}
where the constants in \eqref{sys4} are
\begin{align} 
\tilde a_1 = \frac{a -  r_{\rm g} c_{\infty}}{l_{\rm c}} - g - \tilde{r}_{\rm g} , \quad
\beta =   \frac{D}{l_{\rm c}} , \quad \kappa = \frac{r_{\rm g}}{l_{\rm c}} . 
\end{align} 
In this model, $c(x, t)$ is the tubulin concentration in the axon, varying with the spatial variable $x$ over time $t$, while $q$ represents the combined tubulin flux and concentration at soma. The axon length is denoted $l(t),$ which is the distance between the soma and the growth cone. The subscripts $s$ and $c$ are used for the soma and the growth cone, respectively. This notation can be seen in Fig. \ref{fig:1b}. The parameters, $D$, $a$ and $g$ in \eqref{sys1} are tubulin diffusivity, velocity, and degradation constants, respectively. The parameters in \eqref{sys4} and \eqref{sys5} include $l_{\rm c}$ representing the growth ratio, $\tilde{r}_{\rm g}$ denoting the reaction rate of the microtubules production process, $c_{\infty}$ as the equilibrium of tubulin concentration in the cone, and $r_{\rm g}$ serving as a lumped parameter. Detailed descriptions and derivations of $r_{\rm g}$ and other parameters are available in \cite{diehl2014one}.

\begin{figure}[t]
\centering
\includegraphics[width=0.5\linewidth]{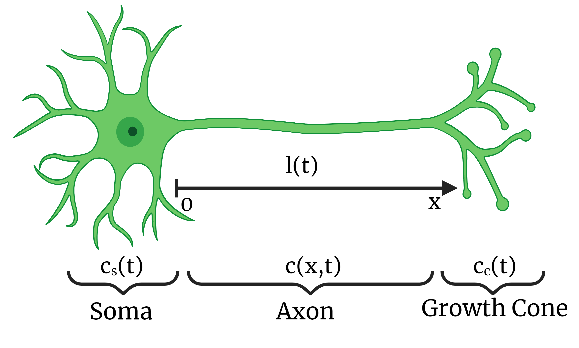}
  \caption{Schematic of neuron and state variables}
  \label{fig:1b} 
\end{figure}

\subsubsection{The steady-state solution and reference error system} We derive a steady-state solution for the concentration, corresponding to a desired axon length $l_{\rm s}$, by setting the time derivatives in \eqref{sys1}, \eqref{sys4}, and \eqref{sys5} to zero. The steady-state spatially distributed steady-state tubulin concentration is
\begin{align}
 \label{ceq} 
c_{\rm eq}(x) = c_{\infty} \left( K_{+} e^{\lambda_+ (x - l_{\rm s})} + K_- e^{\lambda_{-} (x - l_{\rm s}) } \right),
\end{align}
where
\begin{align}
    \lambda_{\pm} =& \frac{a \pm \sqrt{a^2 + 4 D g}}{2 D}, \quad K_{\pm} =  \frac{1}{2} \pm  \frac{a  - 2 g l_{\rm c} }{2 \sqrt{a^2 + 4 D g}},
\end{align}
and the steady-state input for the combination of tubulin flux and concentration in the soma is
\begin{small}
\begin{align}
    q_{\rm s}^* = - c_{\infty} \left( K_{+}(1+\lambda_+)  e^{ - \lambda_+ l_{\rm s}} + K_-(1+\lambda_-)  e^{ - \lambda_{-} l_{\rm s} } \right).
\end{align}
\end{small}

\noindent The reference error system relative to \eqref{sys1}-\eqref{sys5} is given by the dynamics of errors given below
\begin{align}
    \label{eqn:def-u-sys1}
&u_t(x,t) = D u_{xx}(x,t) - a u_x(x,t) - g u(x,t) , \\
&u_x(0,t)+u (0,t) =  U(t), \label{eqn:def-U-cont}\\
&u(l(t),t) =h(X(t)) , \\
 &\dot{X}(t) =  A X(t) + f(X(t)) + B u_x(l(t), t), 
 \label{eqn:def-u-sys4}
\end{align}
where the reference error states, $u(x,t)$, $z_1(t)$ and $z_2(t)$, and the reference error input $U(t)$ are defined as
\begin{align}
u(x,t) =& c(x,t) - c_{\rm eq}(x), \quad U(t) =  - ( q_{\rm s}(t) - q_{\rm s}^*) \\
z_{1}(t) =& c_{\rm c}(t) - c_{\infty}, \quad z_2(t) = l(t) - l_{\rm s}.
\end{align}
Here, $X$ is a state vector in $\mathbb{R}^2$,  given  as
\begin{align} \label{xdef}
    X(t)=[ z_1(t) \quad z_2(t)]^\top . 
\end{align}
The parameters and the functions in \eqref{eqn:def-u-sys1}-\eqref{eqn:def-u-sys4} are defined as follows:
\begin{align}
    \label{AB-def} 
 A &= \left[ 
 \begin{array}{cc}
 \tilde a_1 &  -\beta \tilde{a}_2 \\
 r_{\rm g} & 0
 \end{array}  
 \right] , \quad B =  \left[ 
 \begin{array}{c}
 - \beta \\
 0
 \end{array}  
 \right] ,\\
 f(X(t)) &=  - \kappa z_1^2(t)+\beta f_1(z_2(t)), \\
 h(X(t)) &=  z_1(t) + \tilde h (z_2(t)) , \\
 \tilde{h}(z_2(t))&=  c_{\infty}\left(1 - K_{+} e^{\lambda_+ z_2(t)} - K_- e^{\lambda_{-} z_2(t) }\right). \\
 \tilde{a}_2=&c_{\infty}\left(\lambda_+^2K_++\lambda_-^2K_-\right), \\
f_1(z_2(t))=&-c_{\infty}\left(  K_{+}\lambda_+ e^{\lambda_+ z_2(t)} + K_- \lambda_{-}e^{\lambda_{-} z_2(t)} \right)\nonumber \\
    &+\tilde{a}_2z_2(t)+c_{\infty}\frac{a-gl_{\rm c}}{D}.
\end{align}

\subsection{Control Law Design} 
\label{sec:control} 
\subsubsection{Linearization of the finite-dimensional part of the cascade system }
We begin by linearizing the nonlinear ODEs defined in \eqref{eqn:def-u-sys4} around zero states:
\begin{align}
    \label{ulin-PDE}
u_t(x,t) =& D u_{xx}(x,t) - a u(x,t) - g u(x,t) , \\
u_x(0,t)+u(0,t) = & U(t), \label{ulin-BC1} \\
\label{linreferr3}u(l(t),t) =&H^\top X(t)  , \\
\dot{X}(t) = & A_1  X(t) + B u_x (l(t), t), \label{ulin-ODE}
\end{align}
where the vector $H \in \mathbb{R}^2$ is defined as
\begin{align}
  A_1 &= \left[ 
 \begin{array}{cc}
 \tilde a_1 & \tilde a_3 \\
 r_{\rm g} & 0
 \end{array}  
 \right], \color{black}
 ~ H = \left[1 \quad - \frac{(a-gl_{\rm c}) c_{\infty}}{D}\right]^\top ,  \label{C-def}  
\end{align}
where $\tilde{a}_3=\frac{a^2+Dg-agl_{\rm c}}{D^2}$.
By applying the following backstepping transformation  \begin{align}
 \label{bkst}
w(x,t) = & u(x,t) - \int_x^{l(t)} k(x,y) u(y,t) dy \notag\\
& - \phi(x - l(t))^\top X(t), 
\end{align} 
we can map the linearized reference error system to a desired target system which is
\begin{align} 
\label{tar-PDE} w_t(x,&t) = D w_{xx} (x,t) - a w_x(x,t) - g w(x,t) \notag\\
&- \dot l(t) F(x,X(t))  , \\
\label{tar-BC1} w_x(0,t)+&w(0,t) =  -\frac{1}{D}\left(H-\epsilon\right)^\top Bu(0,t),\\
\label{tar-BC2}w(l(t)&,t) =\epsilon^\top X(t) , \\
\label{tar-ODE} \dot{X}&(t) =  (A_1 + BK^\top) X(t) + B w_x(l(t), t),
 \end{align}
with the redundant nonlinear term $F(x,X(t)) \in R$ in \eqref{tar-PDE} is described as
\begin{align} \label{F-def} 
     F(x,X(t))= \left(\phi'(x-l(t))^T-k(x, l(t)) C^T \right) X(t),
\end{align}
where $K\in\mathbb{R}^2$ is chosen as
\begin{align} 
 k_1 > \frac{\tilde a_1}{\beta} , \quad k_2 > \frac{\tilde{a}_3}{\beta} .
 \end{align} 
to make $A_1+BK$ Hurwitz and $\epsilon\in\mathbb{R}^2$ will be chosen in the stability analysis.

The approach for obtaining gain kernels in \eqref{bkst}, namely $k(x,y)$ and $\phi(x)$, is detailed in \cite{demir2021neuroncontrol}. Simply, $k(x,y)$ and $\phi(x)$ are obtained as:
\begin{align}
    &k(x,y)= -\frac{1}{D}\phi(x-y)^\top B,
    \label{kstar} \\ &\phi(x)^\top=\begin{bmatrix}(H-\epsilon)^\top & K^\top-\frac{1}{D}H^\top BH^\top\end{bmatrix}e^{N_1x}\begin{bmatrix} I \\ 0
\end{bmatrix},
\label{phix}
\end{align}
where 
the matrix $N_1 \in R^{4 \times 4}$ is defined as 
\begin{align}
    N_1=\begin{bmatrix}0 & \frac{1}{D}\left(gI+A+\frac{a}{D}BH^\top\right)\\ I &\frac{1}{D}\left(BH^\top+aI\right)\end{bmatrix}.
\end{align}
The inverse transformation is presented as follows:
\begin{align}
    u(x,t)=&w(x,t)+\int_{x}^{l(t)}q(x,y)w(y,t)dy\nonumber \\
    &+\varphi(x-l(t))^\top X(t)
\end{align}
Detailed solutions for the gain kernels are provided in \cite{demir2021neuroncontrol}.
\subsubsection{Continuous-time and sampled-data control law}
By taking the spatial derivative of the transformation and substituting $x=0$ into both the backstepping transformation and its spatial derivative, and setting boundary condition \eqref{tar-BC1}, the control law is derived as
\begin{align}
  U(t)= -\frac{1}{D}\int_0^{l(t)}p(x)Bu(x,t)dx+p(l(t))X(t),
\label{real-input}
\end{align}
where
\begin{equation}
    p(x) = \phi'(-x)^\top+\phi(-x)^\top.
    \label{eqn:def-px}
\end{equation}
The system outlined in \eqref{sys1}-\eqref{sys5}, with the continuous-time controller input \eqref{real-input}, is locally exponentially stable in the $L^2$-norm sense, as demonstrated in \cite{demir2023event}. To develop the periodic event-triggered control mechanism, the CTC input is sampled at discrete intervals, which holds it constant between events. This approach yields the following sampled-data control.
\begin{align}
     U_{k}^{\omega}(t):=U(t_k^{\omega}),
     \label{eqn:def-Utj}
 \end{align}
where is employed at
 \begin{align}
     u_x(0,t)+u(0,t)=U_k^{\omega}(t).
 \end{align}
for $\forall t\in [t_k^{\omega},t_{k+1}^{\omega})$, $k\in\mathbb{N}$ with the increasing time sequence, $I^{\omega}=\{t_k^{\omega}\}_{k\in\mathbb{N}}$, where $t_0^{\omega}=0$ and $\omega=\{``c",``p"\}$. The notations $``c"$ and $``p"$ represent CETC and PETC, respectively. It's important to note that at each sampling time, the control input is sampled from \eqref{real-input}, in other word, the sampled-data control law is the emulation of the continuous-time controller that is to be implemented in a   Zero-Order Hold fashion.

\section{Continuous-time Event Triggered control}
 In this section, we provide a summary of a CETC design detailed in \cite{demir2023event}.
 \begin{definition} Continuous-time event-triggering consists of two stages: the occurrence of the event and the application of the control signal when the event occurs. These steps are 
 \begin{enumerate}
     \item Detection of the time that event occurs: The set  $I^c=\{t^c_0, t^c_1, t^c_2, \ldots\}$ where $t^c_0=0$, represents an increasing sequence of time instances where the events occur.  This set is generated using the following rules. (i) If $S(t, t_k^c) = \emptyset$, events occur at times $\{t_0, \ldots, t_k^c\}$. (ii) If $S(t,t_k^c)\neq\emptyset$, the subsequent event time is determined as $t_{k+1}^c=\inf\left(S(t,t_k^c)\right)$ where
            \begin{align}
               S(t,t_k^c)=\{t\in R_+|t>t_k^c\wedge d^2(t)>-\gamma m(t)\} 
               \label{eqn:def-S}
            \end{align}
            for all $t\in[t_k^c,t_{k+1}^c)$, $d(t)$ is given as 
\begin{align}
d(t)=U(t)-U_k^{c}(t)
\label{eqn:def-dt}
\end{align} and $m(t)$ satisfies the ODE
\begin{align}
    \dot{m}&(t)=-\eta m(t)+\rho d(t)^2-\beta_1 X(t)^2-\beta_2X(t)^4\nonumber \\
    &-\beta_3X(t)^6-\beta_4|w(0,t)|^2-\beta_5 ||w(x,t))||^2.
    \label{eqn:def-m}
\end{align}

The event-triggering design parameters are $\sigma\in (0,1)$, $\gamma>0$, $\eta>0$, and $\beta_i$ and $\rho$ are selected according to the specifications outlined in \cite{demir2023event} as
\begin{align}
    \rho\geq \frac{d_1^2D}{\delta_1}, \quad
    \beta_i=\frac{\alpha_i}{\gamma(1-\sigma)},
    \label{eqn:def-Bi}
\end{align}
where $\delta_1=2D^2$.
\item Event-based control corresponding to the sampled-data control law \eqref{eqn:def-Utj}
where $p(x)$ is defined in \eqref{eqn:def-px}.
 \end{enumerate}
  \end{definition}
\begin{remark}\label{remark1}
The triggering mechanism defined by \eqref{eqn:def-S} and \eqref{eqn:def-m} has the property of $d^2(t)\leq \gamma m(t)$
and $m(t)>0,$ $\forall~t\in[0,\sup\{I^c\})$ as detailed in \cite{demir2023event}. 
\end{remark}
\begin{remark}
   Considering the increasing set of event-times $\{t_k^c\}_{k\in\mathbb{N}}$ with $t_0=0$, the following bound is obtained for the  time derivative of the input holding error
   \begin{align}
        \dot{d}^2(t)\leq& \rho_1d^2(t)+\alpha_1X(t)^2+\alpha_2X(t)^4+\alpha_3X(t)^6\nonumber \\
        &+\alpha_4w(0,t)^2+\alpha_5||w(x,t)||^2,
    \end{align}
 where the   parameters, $\rho_1, \alpha_1, \alpha_2, \alpha_3, \alpha_4, \alpha_5$ are given by   
    \begin{small}
    \begin{align}
        \rho_1&=7|p(0)B|^2, \label{eqn:def-rho1}\\
    \alpha_1&=\frac{21}{2}\left|\frac{1}{D}\zeta(y)B\int_0^{l(t)}\varphi(x-l(t))^\top dx\right|^2+28(p(0)Bp(l(t)))^2\nonumber \\
    +&21\left(\left(p(0)\left(1-\frac{a}{D}\right)+\dot{p}(0)\right)B\right)^2(\varphi(0)^\top)^2+28(p(l(t))A)^2\nonumber \\
    +&14\left(\left|\dot{p}(l(t))+\frac{a}{D}p(l(t))+\frac{r_{\rm g}}{D}e_1p(l(t))\right|BH^\top\right)^2, \label{eqn:alpha1}\\
    \alpha_2&=7\left(r_{\rm g}e_1 \dot{p}(l(t))+2k_n\left|\dot{p}(l(t))+\frac{a}{D}p(l(t))\right|B+p(l(t))\kappa\right)^2\\
    \alpha_3&=7\left(\frac{2k_n}{D}r_{\rm g}e_1p(l(t))B \right)^2+28\left(k_mp(l(t))\right)^2, \\
    \alpha_4&=21\left(\left(p(0)\left(1-\frac{a}{D}\right)+\dot{p}(0)\right)B\right)^2,\\
    \alpha_5&=7\left|\frac{1}{D}\zeta(y)B\right|^2\left(\frac{9}{2}+\frac{9}{2}\left(\int_{0}^{l(t)}\int_{x}^{l(t)}q(x,y)^2dydx\right)\right) \nonumber \\
    +&21\left(\left(p(0)\left(1-\frac{a}{D}\right)+\dot{p}(0)\right)B\right)^2\bar{G}(l(t))^2, \label{eqn:alpha5} \\
     &\zeta(y):=\int_0^{l(t)}D\ddot{p}(y)-a\dot{p}(y)+gp(y)-p(0)Bp(y)dy, \\
     &\bar{G}(l(t)):=\int_0^{l(t)}q(0,x)dx.
    \end{align}
    \end{small}
\end{remark}

\begin{theorem}\cite{demir2023event}
    For the event-triggered mechanism described in \eqref{eqn:def-Utj}-\eqref{eqn:def-S}, the set of event-times \(\{t_k^c\}_{k \in \mathbb{N}}\) ensures that the function \(\Gamma^c(t) := d(t)^2 - \gamma m(t)\) remains non-positive for all $t \in [t_k^c, t_{k+1}^c)$, where \(k \in \mathbb{N}\).
\end{theorem}
The proof of this theorem and the following results are given in \cite{demir2023event} and the following hold:
\begin{enumerate}
    \item The set of event-times $\{t^c_{k}\}_{k\in\mathbb{N}}$ with triggering mechanism \eqref{eqn:def-Utj}-\eqref{eqn:def-S} and with the design parameters  specified in \eqref{eqn:def-rho1}-\eqref{eqn:alpha5}, ensures that Zeno behavior does not occur. This is because there exists a minimal dwell-time, $\tau>0$, between two execution times, given by
    \begin{align}
            \tau=\int_{0}^{1}\frac{1}{a_1s^2+a_2s+a_3}ds,
            \label{eqn:tau-new}
    \end{align}
    where 
    \begin{align}
        a_1&=\rho\sigma\gamma>0, \label{def:dat-cons1}\\
        a_2&=1+2\rho_1+(1-\sigma)\rho+\eta>0,  \\ 
        a_3&=(1+\rho_1+\gamma(1-\sigma)\rho+\eta)\frac{1-\sigma}{\sigma}>0. \label{def:tau-cons3}
    \end{align}
    \item Given  an initial condition $m(0)<0,$ the variable  $m(t)$ governed by \eqref{eqn:def-m}, satisfies $m(t)<0$  for all $t>0$.
    \item The closed-loop system \eqref{sys1}-\eqref{sys5}, along with the event-triggered mechanism \eqref{eqn:def-Utj}, locally exponentially converges to the desired axon length in the $L^2$-sense.
\end{enumerate}

In the next section, we propose a periodic event-triggering mechanism.

\vspace{-0.25em}

\section{Periodic Event Triggering Mechanism}
In this section, we propose a periodic event-triggering mechanism for axonal growth.

\begin{definition} Consider the event-triggering function $\Gamma^p(t)$, which undergoes periodic evaluation with a period of $h > 0$. The PETC  that generates the events are characterized by two parts:
\begin{enumerate}
    \item The event-trigger mechanism: A periodic event-trigger that determines the event times
    {
\setlength{\abovedisplayskip}{1pt}
\setlength{\belowdisplayskip}{1pt}
    \begin{align}
         t_{k+1}^p=\inf\{t\in \mathbb{R}_+|t> t_k^p,&~\Gamma^p(t)>0,~t=nh,\nonumber \\ &~h>0,~n\in\mathbb{N}\},
        \label{eqn:PETC-con}
    \end{align}}
with $t_0^p=0$ where $h$ is sampling period and 
\begin{align}
    \Gamma^p(t)=\upsilon_1d^2(t)-\upsilon_2m(t)
\end{align}
where $\upsilon_1>0$ and $\upsilon_2>0$.
\item The feedback control law that is derived as
    {
\setlength{\abovedisplayskip}{1pt}
\setlength{\belowdisplayskip}{1pt}
\begin{align}
     U_k^p(t)=&-\frac{1}{D}\int_{0}^{l(t_k^p)}p(x)Bu(x,t_k^p)dx+p(l(t_k^p))X(t_k^p)
     \label{eqn:def-U-tildetj}
 \end{align}}
for all $t\in [t_k^p,t_{k+1}^p)$ for $k\in\mathbb{N}$.
\end{enumerate}
\end{definition}
Note that periodicity in the triggering conditions \eqref{eqn:PETC-con}, allows us to monitor the triggering function \emph{periodically} and update the control laws \emph{aperiodically,} removing the continuous monitoring of the PDE-ODE state variables. Then, the boundary condition \eqref{eqn:def-U-cont} becomes 
\begin{align}
    u_x(0,t)+u(0,t)=U(t_k^p).
    \label{eqn:PETC-bnd}
\end{align}

\vspace{-0.5em}
\subsection{Design of the periodic event triggering function $\Gamma^p(t)$}
\textbf{Selection of the sampling period.} 
The sampling period, denoted as $h$, represents the unit of time during which the control input is updated. Let the periodic event-triggered function given  by \eqref{eqn:PETC-con}, along with the boundary condition in \eqref{eqn:PETC-bnd} and the plant dynamics from \eqref{sys1}-\eqref{sys5}, satisfy the condition $\Gamma^p(t) \leq 0$ for all $t$ within the interval $t\in[t_k^p, t_{k+1}^p)$ for $k\in\mathbb{N}$. Hence, it follows that  $m(t)<0$ for all $t>0$. The parameter $h$ is selected to satisfy
\begin{align}
    0<h\leq \tau,
    \label{eqn:def-samp}
\end{align}
where the upper bound, $\tau$, is the minimum inter-event time of the CETC design defined in  \eqref{eqn:tau-new}-\eqref{def:tau-cons3}.

\vspace{-0.25em}

\vspace{-0.25em}

\begin{proposition}
    Under the definition of the periodic event-triggered boundary control \eqref{eqn:PETC-bnd}, with the sampling period $h<\tau$, it holds that 
        {
\setlength{\abovedisplayskip}{1pt}
\setlength{\belowdisplayskip}{1pt}
    \begin{align}
        \Gamma^c(t)\leq \frac{1}{q}\bigg((a+\gamma\rho)d^2(nh)&e^{q(t-nh)}-\gamma\rho d^2(nh)\nonumber \\
        &+q\gamma m(nh) e^{-\eta(t-nh)}\bigg), 
        \label{eqn:lemma2}
    \end{align} }
    for all $t\in [nh,(n+1)h)$ and any $n\in[t_k^p/h,t_{k+1}^p/h)\subset \mathbb{N}$, where $q=1+\eta+\rho_1$    and  $\Gamma^c(t)=d^2(t)-\gamma m(t)$ for $\gamma>0$.
\end{proposition}

\vspace{-0.25em}

\begin{proof}
    Taking the time derivative of $\Gamma^c(t)$ in $t\in [nh,(n+1)h)$ and $n\in[t_k^p/h,t_{k+1}^p/h)\subset \mathbb{N}$, one can show that
            {
\setlength{\abovedisplayskip}{1pt}
\setlength{\belowdisplayskip}{1pt}
    \begin{align}
        \dot{\Gamma}^c(t)&=2d(t)\dot{d}(t)-\gamma\dot{m}(t) \\
        &\leq d^2(t)+\dot{d}^2(t)-\gamma \dot{m}(t).
    \end{align}}
\noindent By using Lemma 1, we get
\begin{align}
    \dot{\Gamma}^c(t)\leq& (1+\rho_1+\gamma\rho)d^2(t)-\left(\gamma\beta_1-\alpha_1\right)X(t)^2\nonumber \\
    &-\left(\gamma\beta_2-\alpha_2\right)X(t)^4-\left(\gamma\beta_3-\alpha_3\right)X(t)^6\nonumber \\
    &-\left(\gamma\beta_4-\alpha_4\right)u(0,t)^2-\left(\gamma\beta_5-\alpha_5\right)||u(x,t)||^2 \nonumber \\
    &+\eta\gamma m(t).
\end{align}
By using the definition of $\Gamma^c(t)$, we get
\begin{align}
    \dot{\Gamma}^c(t)\leq& (1+\rho_1+\gamma\rho)\Gamma^c(t)-\left(\gamma\beta_1-\alpha_1\right)X(t)^2\nonumber \\
    &-\left(\gamma\beta_2-\alpha_2\right)X(t)^4-\left(\gamma\beta_3-\alpha_3\right)X(t)^6\nonumber \\
        &-\left(\gamma\beta_4-\alpha_4\right)u(0,t)^2-\left(\gamma\beta_5-\alpha_5\right)||u(x,t)||^2\nonumber\\
        &+\left(\left(1+\rho_1+\gamma\rho\right)\gamma+\eta\gamma\right) m(t).
        \label{eqn:gammafun}
\end{align}
Since $m(t)$ satisfies Remark \ref{remark1} and \eqref{sys1}-\eqref{sys5} with the event-triggered control law \eqref{eqn:def-Utj} is locally exponentially convergen, \eqref{eqn:gammafun} exhibit smooth behavior in the interval $t\in [nh,(n+1)h)$ and for any $n\in[t_k^p/h,t_{k+1}^p/h)\subset \mathbb{N}$. This establishes the existence of a non-negative function $\iota(t)\in C^0((t_k^p,t_{k+1}^p);\mathbb{R}_+)$ such that:
\begin{align}
    \dot{\Gamma}^c(t)=& (1+\rho_1+\gamma\rho)\Gamma(t)-\left(\gamma\beta_1-\alpha_1\right)X(t)^2\nonumber \\
    &-\left(\gamma\beta_2-\alpha_2\right)X(t)^4-\left(\gamma\beta_3-\alpha_3\right)X(t)^6\nonumber \\
        &-\left(\gamma\beta_4-\alpha_4\right)u(0,t)^2-\left(\gamma\beta_5-\alpha_5\right)||u(x,t)||^2\nonumber\\
        &+\left(\left(1+\rho_1+\gamma\rho\right)\gamma+\eta\gamma\right) m(t)-\iota(t),
        \label{eqn:gammadotdef}
\end{align}
for all $t\in [nh,(n+1)h)$ and for any $n\in[t_k^p/h,t_{k+1}^p/h)\subset \mathbb{N}$. Moreover, through the substitution of $d^2(t)=\Gamma^c(t)+\gamma m(t)$, we can rephrase the dynamics of $m(t)$ as follows:
        {
\setlength{\abovedisplayskip}{1pt}
\setlength{\belowdisplayskip}{1pt}
\begin{align}
    \dot{m}(t)=&-\rho\Gamma^c(t)-\left(\eta +\rho\gamma\right)m(t)+\beta_1 X(t)^2+\beta_2X(t)^4\nonumber \\
    &+\beta_3X(t)^6+\beta_4|u(0,t)|^2+\beta_5 ||u(x,t))||^2,
    \label{eqn:mdotdef}
\end{align}}
for all $t\in [nh,(n+1)h)$ and for any $n\in[t_k^p/h,t_{k+1}^p/h)\subset \mathbb{N}$. Subsequently, by combining \eqref{eqn:gammadotdef} with \eqref{eqn:mdotdef}, we can derive the subsequent system of ODEs:
        {
\setlength{\abovedisplayskip}{1pt}
\setlength{\belowdisplayskip}{1pt}
\begin{align*}
    \dot{r}(t)=A_1r(t)+v(t),
\end{align*}}
where
\begin{align}
    &r(t)=\left[\begin{array}{c}
         \Gamma^c(t)  \\
          m(t)
    \end{array}\right], A_1=\left[\begin{array}{cc}
         q-\eta+\gamma\rho & \gamma\left(q+\gamma\rho\right)\\
          -\rho & -\eta -\rho\gamma
    \end{array}\right], \nonumber \\
    &v(t)=\left[\begin{array}{c}
         f_1(t) \\
         f_2(t)
    \end{array}\right] ,
    \label{eqn:ltv-ODE}
\end{align}
where
\begin{align}
    f_1(t)=&-\left(\gamma\beta_1-\alpha_1\right)X(t)^2-\left(\gamma\beta_2-\alpha_2\right)X(t)^4\nonumber \\
    &-\left(\gamma\beta_3-\alpha_3\right)X(t)^6-\left(\gamma\beta_4-\alpha_4\right)u(0,t)^2\nonumber \\
        & -\left(\gamma\beta_5-\alpha_5\right)||u(x,t)||^2-\iota(t), \\
    f_2(t)=&\beta_1 X(t)^2+\beta_2X(t)^4+\beta_3X(t)^6+\beta_4|u(0,t)|^2 \nonumber \\
    &+\beta_5 ||u(x,t))||^2,
\end{align}
and 
    {
\setlength{\abovedisplayskip}{1pt}
\setlength{\belowdisplayskip}{1pt}
\begin{align}
    q=1+\eta+\rho_1.
    \label{eqn:def-q}
\end{align}}
The solution to \eqref{eqn:ltv-ODE} for all $t\in [nh,(n+1)h)$ and for any $n\in[t_k^p/h,t_{k+1}^p/h)\subset \mathbb{N}$ can be expressed as:
    {
\setlength{\abovedisplayskip}{1pt}
\setlength{\belowdisplayskip}{1pt}
\begin{align}
    r(t)=e^{A_1(t-nh)}r(nh)+\int_{nh}^{t}e^{A_1(t-\xi)}v(\xi)d\xi,
\end{align}}
which gives us
    {
\setlength{\abovedisplayskip}{1pt}
\setlength{\belowdisplayskip}{1pt}
\begin{align}
    \Gamma^c(t)=C_1e^{A_1(t-nh)}r(nh)+C_1\int_{nh}^{t}e^{A_1(t-\xi)}v(\xi)d\xi,
    \label{eqn:gamma-explicit}
\end{align}}
where $C_1=\left[
         1 \quad 0
    \right]$. Since matrix $A_1$ has two distinct eigenvalues, we can diagonalize the matrix exponential $e^{A_1t}$ as it is defined in \cite{rathnayake2023periodic}. Thus, we can derive the second part of \eqref{eqn:gamma-explicit} as
        {
\setlength{\abovedisplayskip}{1pt}
\setlength{\belowdisplayskip}{1pt}
\begin{align}
    C_1e^{A_1(t-\xi)}v(\xi)=\left[\begin{array}{cc}
         g_1(t) & g_2(t) 
    \end{array}\right]\left[\begin{array}{c}
         f_1(t) \\ f_2(t) 
    \end{array}\right],
    \label{eqn:LTVgamma}
\end{align}}
where
    {
\setlength{\abovedisplayskip}{1pt}
\setlength{\belowdisplayskip}{1pt}
\begin{align}
    g_1(t)&=\frac{q+\gamma\rho}{q}e^{(1+\rho_1)t}-\frac{\rho\gamma}{q} e^{-\eta t } , \\
    g_2(t)&=\frac{\gamma\left(q+\gamma\rho\right)}{q}\left(e^{(1+\rho_1)t}-e^{-\eta t }\right).
\end{align}}
Since we have the following relationship
        {
\setlength{\abovedisplayskip}{1pt}
\setlength{\belowdisplayskip}{1pt}
\begin{align}
    1+\eta+7|p(0)B|^2>0,
\end{align}}
we can get
    {
\setlength{\abovedisplayskip}{1pt}
\setlength{\belowdisplayskip}{1pt}
\begin{align}
    g_1(t)&=\frac{1}{q}\left(-\gamma\rho+(q+\rho\gamma)e^{qt}\right)e^{-\eta t},\\
    g_2(t)&=\frac{\gamma(q+\gamma\rho)}{q}\left(-1+e^{qt}\right)e^{-\eta t}.
\end{align}}
It's apparent that $g_1(t)$ remains positive for $t>0$. Furthermore, considering the relation
\eqref{eqn:def-Bi}, and using ascending order of triggering times that is the solution of \eqref{eqn:tau-new} is represented by
        {
\setlength{\abovedisplayskip}{1pt}
\setlength{\belowdisplayskip}{1pt}
\begin{align}
    \tau=\frac{1}{q}\ln{\left(1+\frac{\sigma q}{(1-\sigma)(q+\gamma\rho)}\right)},
\end{align} }
one can show that
        {
\setlength{\abovedisplayskip}{1pt}
\setlength{\belowdisplayskip}{1pt}
\begin{small}
\begin{align}
    C_1&e^{A_1(t-\xi)}v(\xi)=\frac{\alpha_1(q+\gamma\rho)}{q}\left(e^{q\tau}-e^{q(t-\xi)}\right)e^{-\eta(t-\xi)}X(t)^2 \nonumber \\
    &+\frac{\alpha_2(q+\gamma\rho)}{q}\left(e^{q\tau}-e^{q(t-\xi)}\right)e^{-\eta(t-\xi)}X(t)^4\nonumber \\
    &+\frac{\alpha_3(q+\gamma\rho)}{q}\left(e^{q\tau}-e^{q(t-\xi)}\right)e^{-\eta(t-\xi)}X(t)^6\nonumber \\
    &+\frac{\alpha_4(q+\gamma\rho)}{q}\left(e^{q\tau}-e^{q(t-\xi)}\right)e^{-\eta(t-\xi)}u(0,t)^2\nonumber \\
    &+\frac{\alpha_5(q+\gamma\rho)}{q}\left(e^{q\tau}-e^{q(t-\xi)}\right)e^{-\eta(t-\xi)}||u(x,t)||^2.
    \label{eqn:ltv-last}
\end{align}
\end{small}
}
\noindent Given the stipulated intervals $nh\leq \xi \leq t \leq (n+1)h$, and $h\leq\tau$, upon thorough examination of \eqref{eqn:ltv-last}, it emerges that the inequality $(\gamma\beta_i-\alpha_i)g_1(t-\xi)-\beta_ig_2(t-\xi)>0$ satisfied for all $i=1,2,3,4,5$. This observation prompts us to  establish $C_1e^{A_1(t-\xi)}v(\xi)$ which holds for all $t$ and $\xi$ within the range of $nh\leq \xi \leq t \leq (n+1)h$, and for $n\in[t_k^p/h,t_{k+1}^p/h)\subset \mathbb{N}$. Taking this observation into account alongside \eqref{eqn:gamma-explicit}, we can derive the following expression for $t\in[nh,(n+1)h)$:
\begin{align}
    \Gamma^c(t)
    &\leq \frac{1}{q}\left(-\gamma(q+\gamma\rho)m(nh)-\gamma\rho\Gamma^c(nh)\right.\nonumber \\
    &\left.+(q+\gamma\rho)\left(\Gamma^c(nh)+\gamma m(nh)\right)e^{q(    t-nh)}\right).
    \label{eqn:Gammat-last}
\end{align}
Upon performing the substitution $\Gamma^c(nh)$ into \eqref{eqn:Gammat-last}, we are able to derive the inequality \eqref{eqn:lemma2} which is valid for all $t \in [nh, (n+1)h)$. This concludes the proof.
\end{proof}
Building upon Lemma 2, the update time for the control input can be determined by identifying when the subsequent condition is met for any $t\in[nh,(n+1)h)$, thereby challenging the positive definiteness of $\Gamma^c(t)$.

\begin{align}
(q+\gamma\rho)d^2(nh)e^{q(t-nh)}-\gamma\rho d^2(nh)+q\gamma m(nh)>0,
\end{align}

\noindent Thus, one can choose this condition as $\Gamma^p(t)$ such that
\begin{align}
    \Gamma^p(t)=(q+\gamma\rho)e^{qh}d^2(t)-\gamma\rho d^2(t)+q\gamma m(t),
    \label{eqn:def-Tilde-Gammat}
\end{align}
which completes the design process.

\begin{theorem}Let the design parameters, $\rho$, $\rho_1$ and $\beta_i$ as defined in \eqref{eqn:def-Bi}, \eqref{eqn:def-rho1}-\eqref{eqn:alpha5}, set the sampling rate in accordance with \eqref{eqn:def-samp}, let $\gamma,\eta>0$ and $\sigma\in (0,1)$. Let us consider the periodic event-triggering mechanism \eqref{eqn:PETC-con}-\eqref{eqn:def-U-tildetj} with the $\Gamma^p(t)$ as defined in \eqref{eqn:def-Tilde-Gammat} which generates the increasing sequence of times $\{t_k^p\}_{k\in\mathbb{N}}$ with $t_0^p=0$. Then, for $\Gamma^c(t)$ and $m(t)$ with $m(t)>0$, it holds that $\Gamma^c(t)\leq 0$ and $m(t)> 0$ for all $t>0$.
\end{theorem}
\begin{proof}
Due to space constraints, we omit this proof, which can be stated following the steps of the proof of Theorem 2 in \cite{rathnayake2023periodic}.
\end{proof}
\color{black}

\vspace{-0.5em}

\subsection{Local exponential convergence under PETC}
In order to prove that the closed-loop system \eqref{sys1}-\eqref{sys5} with the control law \eqref{eqn:def-Utj} and the periodic event-triggering mechanism \eqref{eqn:PETC-con} and \eqref{eqn:def-Tilde-Gammat}, is locally exponentially convergent, we first obtain the following target system  by applying transformation \eqref{bkst}  

        {
\setlength{\abovedisplayskip}{1pt}
\setlength{\belowdisplayskip}{1pt}
\begin{small}
\begin{align} 
\label{tar-PDE-nonlin}  &w_t (x,t) = D w_{xx} (x,t) - a w_x (x,t) - g w(x,t) - \dot l(t) F(x,X(t))\notag\\
& \quad \quad \quad -\phi(x-l(t))^\top f(X(t))-G(x,l(t))h^*(X), \\
\label{tar-BC1-nonlin} &w_x(0,t)+w(0,t) = d(t)-\frac{1}{D}\left(H-\epsilon\right)^\top Bu(0,t),\\
\label{tar-BC2-nonlin}&w(l(t),t) =h^*(X(t))+\epsilon^\top X(t) , \\
\label{tar-ODE-nonlin} &\dot{X}(t) =  (A + BK) X(t) +f(X(t)) + B w_x (l(t), t),  
 \end{align}
 \end{small}}
 
\noindent where $G(x,l(t)):=\left(\phi'(x-l(t))^\top +\frac{a}{D}\phi(x-l(t))^\top \right) B.$  Using the  transformation below 
        {
\setlength{\abovedisplayskip}{1pt}
\setlength{\belowdisplayskip}{1pt}
\begin{align}
    \varpi(x,t)=w(x,t)-h^*(X(t))
    \label{eqn:def-2ndtrans}
\end{align}}
converts
\eqref{tar-PDE-nonlin}-\eqref{tar-ODE-nonlin} into
        {
\setlength{\abovedisplayskip}{1pt}
\setlength{\belowdisplayskip}{1pt}
\begin{align}
\label{eqn:varpi1}
    &\varpi_t(x,t)=D \varpi_{xx} (x,t) - a \varpi_x (x,t) - g \varpi(x,t) \nonumber \\
&\quad   +gh^*(X(t)) - \dot l(t)F(x,X(t))-\dot{h}^*(X(t)) B \varpi_x (l(t), t)  \nonumber \\
&\quad   -\phi(x-l(t))^\top f(X(t)) -G(x,l(t))h^*(X) \nonumber \\
&\quad  -\dot{h}^*(X(t))\left((A + BK) X(t) +f(X(t))\right), \\
&\varpi_x(0,t)+\varpi(0,t)=d(t)-\frac{1}{D}\left(H-\epsilon\right)^\top Bu(0,t)\nonumber \\
&\quad \quad \quad \quad \quad \quad \quad \quad +h^*(X(t)), \\
&\varpi(l(t),t)=\epsilon^\top X(t), \\
&\dot{X}(t)=(A + BK) X(t) +f(X(t)) + B \varpi_x (l(t), t).
\label{eqn:varpiend}
\end{align}}
Below, we state the convergence result.

\vspace{-0.5em}

\begin{theorem}
Let the design parameters, $\rho$, $\rho_1$ and $\beta_i$ given as defined in Theorem 2. Consider the periodic event-triggering rule \eqref{eqn:PETC-con}-\eqref{eqn:def-U-tildetj} with the periodic event-triggering function \eqref{eqn:def-Tilde-Gammat} and sampling rate $h$  defined in \eqref{eqn:def-samp}, which generates an increasing event-times $\{t_k^p\}_{k\in\mathbb{N}}$. Assuming the well-posedness, the closed-loop system of \eqref{sys1}-\eqref{sys5} with the boundary control law \eqref{eqn:def-Tilde-Gammat} and \eqref{eqn:def-px} is locally exponentially convergent in $L^2$-norm sense.
\end{theorem}

\vspace{-0.5em}

\begin{proof}
To demonstrate the local convergence of the system, we initially establish the system properties in a non-constant spatial interval as derived in \cite{demir2023event}, outlined as follows:
\begin{align}
    0<l(t)\leq \bar{l}, \quad |\dot{l}(t)|\leq \bar{v}
    \label{eqn:sys-pro}
\end{align}
for some $\bar{l}>l_{\rm s}>0$ and $\bar{v}=\frac{D}{16(D+1)}$.  As demonstrated in Theorem 2 of \cite{demir2023event}, $m(t)<0$ for all $t\in[t_k^p,t_{k+1}^p)$ where $k\in\mathbb{N}$, implying $\Gamma^c(t)\leq 0$ for $t\in[t_k^p,t_{k+1}^p)$. Assuming the well-posedness of the closed-loop system and following the methodology outlined in \cite{demir2023event}, the subsequent Lyapunov functional is considered
\begin{align}
    V(t)=V_1(t)-m(t),\label{w}
\end{align}
where
    \begin{small}
    \begin{align}
        V_1(t)=& d_1\fr{1}{2} \int_0^{l(t)} \varpi(x,t)^2 dx+X(t)^\top \left(d_2P_1+\frac{1}{2}P_2\right) X(t) 
        \label{eqn:def-V}
    \end{align}
    \end{small}

\noindent and $d_1>0$, $d_2>0$, $P_1\succ0$ and $P_2\succeq 0$ are positive definite and positive semidefinite matrices satisfying the Lyapunov equations:
\begin{align}
    &(A + BK^\top )^\top P_1 + P_1 (A + BK^\top ) = - Q_1, \nonumber \\
    &(A + BK^\top )^\top (P_1+P_2) + (P_1+P_2) (A + BK^\top ) = - Q_2 \nonumber
\end{align}
where
\begin{align}
    P_1=\begin{bmatrix}
        p_{1,1} & p_{1,2} \\
        p_{1,2} & p_{2,2} 
    \end{bmatrix}, \quad P_2=\begin{bmatrix} \frac{D\epsilon_1}{\beta}-2p_{1,1} &0 \\ 0 & 0\end{bmatrix}
\end{align}
where we pick $\epsilon\in\mathbb{R}^2$ as 
$\epsilon_1\geq 2l_{\rm c}p_{1,1}$ and  $\epsilon_2 = \frac{p_{1,2}}{l_{\rm c}d_1}$ for some positive definite matrices $Q_1\succ0$ and $Q_2\succ0$. By taking the time derivative of \eqref{eqn:def-V}, applying Poincaré's, Agmon's, and Young's inequalities, we first derive the following expression:
\begin{align}
    \dot{V}\leq &-\alpha^* V +\xi_1 V^{3/2}+\xi_2 V^2+\xi_3V^{5/2}+\xi_4V^{3}
    \label{def-dotVtot3}
\end{align}
where
\begin{align}
   \alpha^*&= \min\left\{\frac{g}{2},\frac{1}{2\lambda_{\min}(P_1+P_2)},\eta\right\},\\
   \xi_1&=\frac{\left(Dd_1|\epsilon \bar{B}|+2d_2\left|P_1^\top B\bar{B}\right|\right)\kappa^2+\frac{d_1r_{\rm g}}{2}(1+L_1)+r_{\rm g}}{d_2^{3/2}\lambda_{\min}(P_1+P_2)^{3/2}}, \label{def:Xi1-new}\\
    \xi_2&=\frac{\Xi_1}{d_2^{2}\lambda_{\min}(P_1+P_2)^2}, \quad 
    \xi_3=\frac{4d_2k_m|P_1|}{d_2^{5/2}\lambda_{\min}(P_1+P_2)^{5/2}},\\
    \xi_4&=\frac{\Xi_2}{d_2^3\lambda_{\min}(P_1+P_2)^3},\label{def:Xi4-new}
\end{align}
 where $d_1$ and $d_2$ are chosen to satisfy
\begin{align}
    d_1&\geq \max\left\{\frac{8\bar{l}\left(D+2\right)+16\bar{l}\beta_4}{D},\frac{4\beta_5+7}{g}\right\}, \\
    d_2&\geq \frac{4}{\lambda_{\min}(Q_2)}\left(Dd_1\left|\epsilon \bar{B}|(A+BK)\right|+\beta_1\right) \nonumber \\
    &+\frac{4}{\lambda_{\min}(Q_2)}\left(\left(D+2+Dd_1+\frac{d_1a}{2}+2\beta_4\right)\frac{2}{\beta^2}\right).
\end{align}
Note that the positive constant parameters are also defined as follows:
\begin{align}
    &F(0,X(t))^2 \leq L_1 X^\top X, \\ &\int_0^{l(t)}\left(\phi(x-l(t))^\top\right)^2dx  \leq  L_{n_2}, 
    \\
    &\int_0^{l(t)}\left(\phi'(x-l(t))^\top B -ak(x,l(t))\right)^2dx  \leq L_{n_3} \label{eqn:Ln3}
    \\
    &\Xi_1=4Dd_1|\epsilon \bar{B}|k_m^2|P_1|^2+8d_2\left|P_1^\top B\bar{B}\right|k_m^2|P_1|^2+\beta_2\nonumber \\
    &\quad +2d_1^2L_{n_3}k_n^2+\frac{d_1^2}{2}L_{n_2}\kappa^2+d_1^2c_{\infty}^2r_{\rm g}^2k_l+8d_2\kappa |P|\beta_5 k_{l}\ \nonumber\\
    &\quad+2d_2\kappa |P|\left(d_1^2\left(\beta^2(1-\epsilon_1)^2\left(1+\bar{G}(l(t))^2\right)+D\right)\right)k_l,
    \\
    &\Xi_2=d_1^2c_{\infty}^2r_{\rm g}^2k_l+\frac{d_1^2}{2}L_{n_2}4k_m^2|P_1|^2+\beta_3 \nonumber \\
    &\quad+\left(d_1^2\left(\beta^2(1-\epsilon_1)^2\left(1+\bar{G}(l(t))^2\right)+D\right)+4\beta_5\right)k_l,
    \\
    &k_l=\max\left\{\left|K_+\lambda_+\right|,\left|K_-\lambda_-\right|\right\}^2, \quad \bar{B}=[-\beta^{-1}~ 0], \\ &k_n=c_{\infty}\max\{K_{+}\lambda_{+}^2,K_{-}\lambda_{-}^2\}, \\  
    &k_m=c_{\infty}\max\{K_{+}\lambda_{+}^3,K_{-}\lambda_{-}^3\}, \\
    &-e^x+x+1\leq x^2 \quad \text{for}  \quad x\leq 1.79.
\end{align}
Given \eqref{def-dotVtot3}, we can demonstrate that within the region $\Omega_1:=\{(\varpi,X) \in L^2 \times \mathbb{R}^2 | V(t) < M_0\}$ where $t\in(t_k^p,t_{k+1}^p)$ for $k\in\mathbb{N}$, there exists a positive constant $M_0>0$ ensuring the satisfaction of the system properties \eqref{eqn:sys-pro}. The existence of such $M_0>0$ is established in Lemma 2 of \cite{demir2021neuroncontrol}. From this result, we have $M_0=\frac{\lambda_{\rm min}(P_1)}{d_2}r^2$ where $r=\min\left\{\frac{\bar{v}}{r_{\rm g}}, l_{\rm s},\bar{l}-l_{\rm s}\right\}$ for $t\in (t_k^p,t_{k+1}^p)$, $k\in\mathbb{N}$. Next, we analyze the convergence within the time interval $t \in (t_k^p, t_{k+1}^p)$ for $k \in \mathbb{N}$, and subsequently for $t\in(0,t)$. Then, a positive constant $M$ exists such that when $V(t_j)<M$, the following norm estimate is valid for $t\in[t_k^p,t_{k+1}^p)$, where $k\in\mathbb{N}$:
\begin{align}
     V(t_{k+1}^p)\leq V(t_k^p)e^{-\frac{\alpha^*}{2}(t_{k+1}^p-t_k^p)}.
     \label{def:eqn:V-norm}
\end{align}
For $M>0$, we define the set $\Omega:=\{(\varpi,X) \in L^2 \times \mathbb{R}^2 | V(t) < M\}$. From Lemma 2 in \cite{demir2021neuroncontrol}, it is clear that if $M\leq M_0$, then $\Omega \subset \Omega_1$ which satisfy the system properties \eqref{eqn:sys-pro} and the norm estimate defined in \eqref{def:eqn:V-norm} for $t\in[t_k^p,t_{k+1}^p)$, where $k\in\mathbb{N}$. Hence, we set $M\leq p^*$, where $p^*$ is a non-zero root of the following polynomial
\begin{align}
    -\alpha^*V+\xi_1 V^{3/2}+\xi_2 V^2+\xi_3V^{5/2}+\xi_4V^{3}=0
\end{align}
for $V>0$. Given that all coefficients of this polynomial are positive, at least one positive root $p^*$ exists. Thus, \eqref{def-dotVtot3} implies
\begin{align}
    \dot{V}\leq -\frac{\alpha^*}{2}V(t)
\end{align}
for $t\in[t_k^p,t_{k+1}^p)$, where $k\in\mathbb{N}$ and $M=\min\{M_0,p^*\}$. The smoothness of $V(t)$ within this interval ensures that $V(t_{k+1}^{p^-})=V(t)$ and $V(t_{k}^{p^+})=V(t_k^p)$, where $t_{k}^{p^+}$ and $t_{k}^{p^-}$ denote the right and left limits of $t=t_k^p$, respectively. Thus, we can have the norm estimate in  \eqref{def:eqn:V-norm}. Then, for any $t\geq 0$ in $t\in [t_k^p,t_{k+1}^p)$ where $k\in\mathbb{N}$, we have
\begin{align}
     V(t)\leq e^{-\alpha^*(t-t_k^p)}V(t_k^p)
    \leq e^{-\alpha^* t}V(0).
\end{align}
Recalling $m(t)<0$ and \eqref{w}, we can write
\begin{align}
    V_1(t)-m(t)\leq e^{-\alpha^* t}V(0)
\end{align}
by applying the comparison principle one can obtain the following norm estimate for the target system $(\varpi,X)$:
    \begin{align}
        &d_1\fr{1}{2} ||\varpi(x)||^2+d_2X(t)^\top \left(P_1+\frac{1}{d_2}P_2\right) X(t) \nonumber \\
    \leq& e^{-\alpha^* t}\left(\fr{d_1}{2} ||\varpi(0)||^2+d_2X(0)^\top \left(P_1+\frac{1}{d_2}P_2\right) X(0)\right) \nonumber \\
    &-e^{-\alpha^* t}m(0)
    \end{align}
Utilizing the invertibility of the transformation \eqref{eqn:def-2ndtrans}, we subsequently prove that the target system $(w,X)$ is also locally exponentially convergent. For the original system $(u,X)$, we leverage the invertibility of the backstepping transformation given in \eqref{bkst}. Consequently, we conclude that the closed-loop system is also locally exponentially convergent. This completes the proof.
\end{proof}

\begin{figure}[t!]
\centering
{       \includegraphics[width=0.9\linewidth]{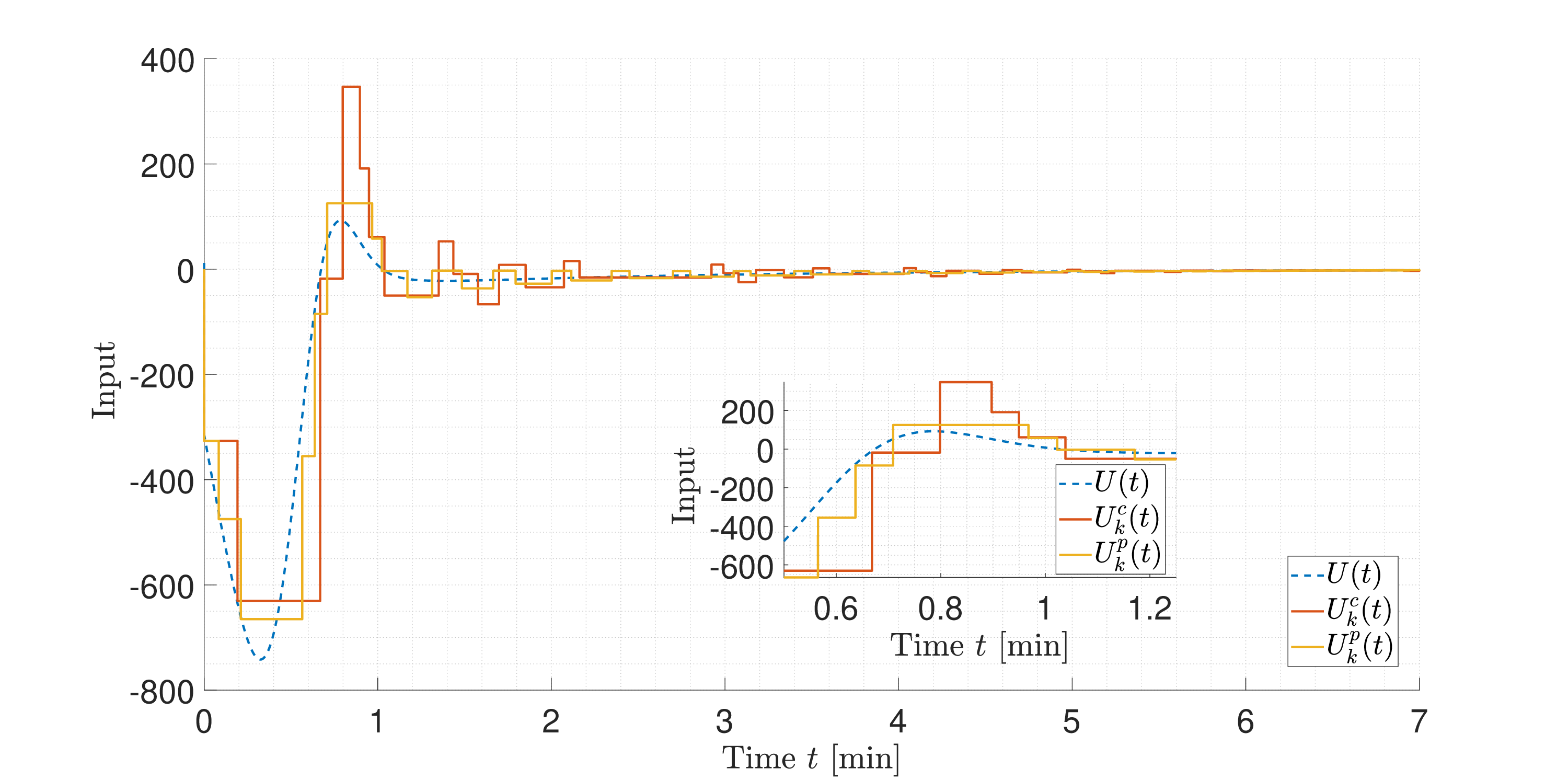}}
  \caption{Comparison between periodic-event triggering control input $U_k^p(t)$, continuous time event triggering control input $U_k^c(t)$ and the continuous control law $U(t)$}
  \label{fig:2} 
\end{figure}

\vspace{-0.5em}

\section{Numerical Simulations}
In this section, we conduct numerical simulations for the system represented by equations \eqref{sys1}-\eqref{sys5}, employing the control law \eqref{real-input} along with the designed periodic event-triggering mechanism \eqref{eqn:PETC-con} utilizing the triggering function \eqref{eqn:def-Tilde-Gammat}. The model parameters are detailed in Table I. Initial conditions are specified as $c_0(x)=1.5c_{\infty}$ for the tubulin concentration along the axon and $l_0=1\mu m$ for the initial axon length. Control gain parameters are set as $k_1=-0.001$ and $k_2=3\times 10^{13}$. The event-triggering parameters are set as follows $m(0)=-0.5$, $\beta_1=2.5\times 10^{8}$, $\beta_2=8\times10^{9}$, $\beta_3=1\times 10^{11}$, $\beta_4=4\times 10^{11}$, $\beta_5=4.5\times 10^{11}$, $\rho=1.5\times 10^{-15}$, $\gamma=1$, $\eta=2$ and $\sigma=0.8$. Moreover, the sampling period for the periodic event-triggering mechanism is selected as $h=0.5~ ms$ which is smaller than the minimal dwell time $\tau\approx0.54~ ms$.

Fig. 2 illustrates the evolution of the continuous-time control input, $U(t)$, the event-triggering control input, $U_k^c(t)$, as defined in \eqref{eqn:def-Utj} with the triggering mechanism given by \eqref{eqn:def-S}-\eqref{eqn:def-m}, and the periodic event-triggering control input, $U_k^p(t)$, as defined in \eqref{eqn:def-U-tildetj} with the triggering condition in \eqref{eqn:PETC-con} and triggering function in \eqref{eqn:def-Tilde-Gammat}. While PETC closely emulates the CETC control input behavior, both PETC and CETC minimized the necessity of control law updates by maintaining comparable performance. In Fig. 3, tubulin concentration, $c(x,t)$, and axon length, $l(t)$, converge to the steady-state solution of the tubulin concentration and the desired axon length. Note that tubulin concentration exhibits smoother changes with the PETC mechanism compared to the CETC mechanism which enhances the practical applicability. 

\begin{table}[!b]
\hfill	\caption{\label{tab:initial}Biological constants and control parameters}
	\centering
	\begin{tabular}{c|c|c|c} 
		\hline
		Parameter  & Value  & Parameter & Value \\
		\hline
		$D$ & $10\times10^{-12}  m^2/s$& $\tilde{r}_{\rm g}$ & $0.053$\\
		$a$& $1\times 10^{-8}  m/s$ & $\gamma$ &  $10^4$\\
		$g$& $5\times 10^{-7} \ s^{-1}$ & $l_{\rm c}$ & $4\mu m$\\
		$r_{\rm g}$& $1.783\times 10^{-5} \ m^4/(mol s)$ & $l_s$ & $12\mu m$ \\
		$c_{\infty}$ &  $0.0119  \ mol/m^3$ & $l_0$& $1\mu m$ \\
 \hline
	\end{tabular}
\end{table}

\begin{figure*}[t]
\begin{center} 
\subfloat[Continous time control input.]
{\includegraphics[width=0.32\linewidth]{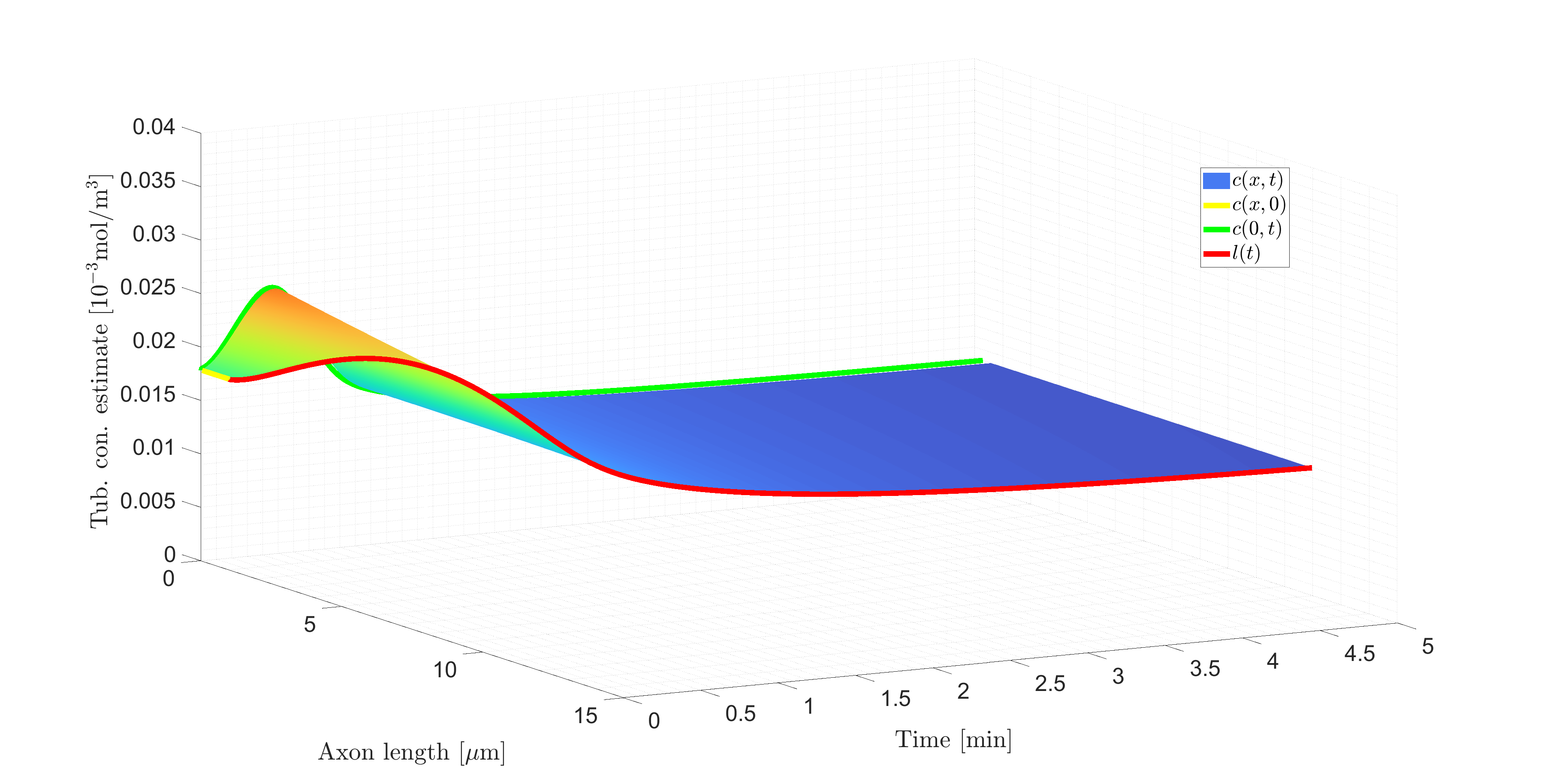}\label{fig:Ut}}\hspace{1mm}
\subfloat[Event-triggered  control input.]
{\includegraphics[width=0.32\linewidth]{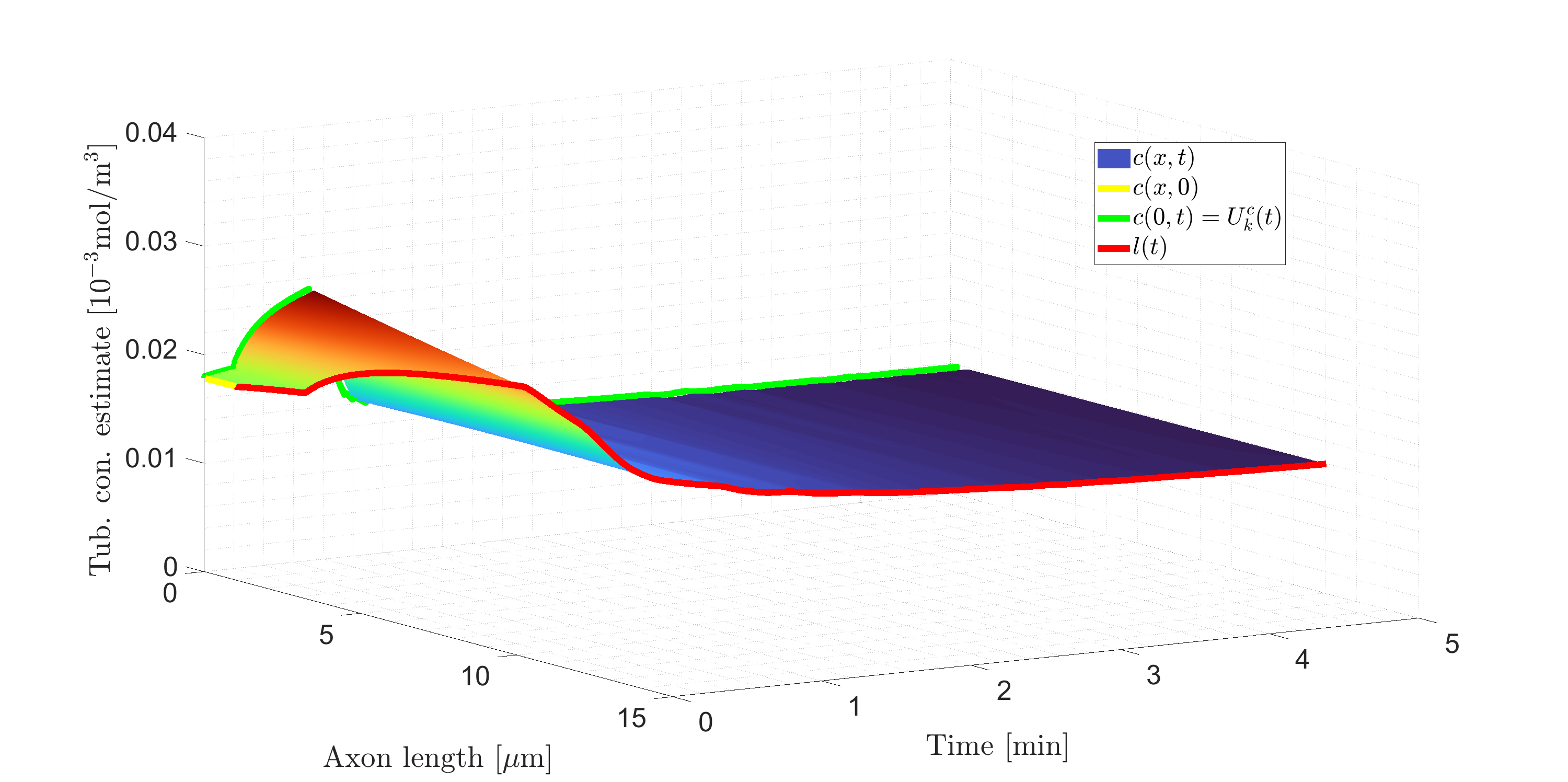}\label{fig:qc}}\hspace{1mm}
\subfloat[Periodic event-triggered control input. ]
{\includegraphics[width=0.32\linewidth]{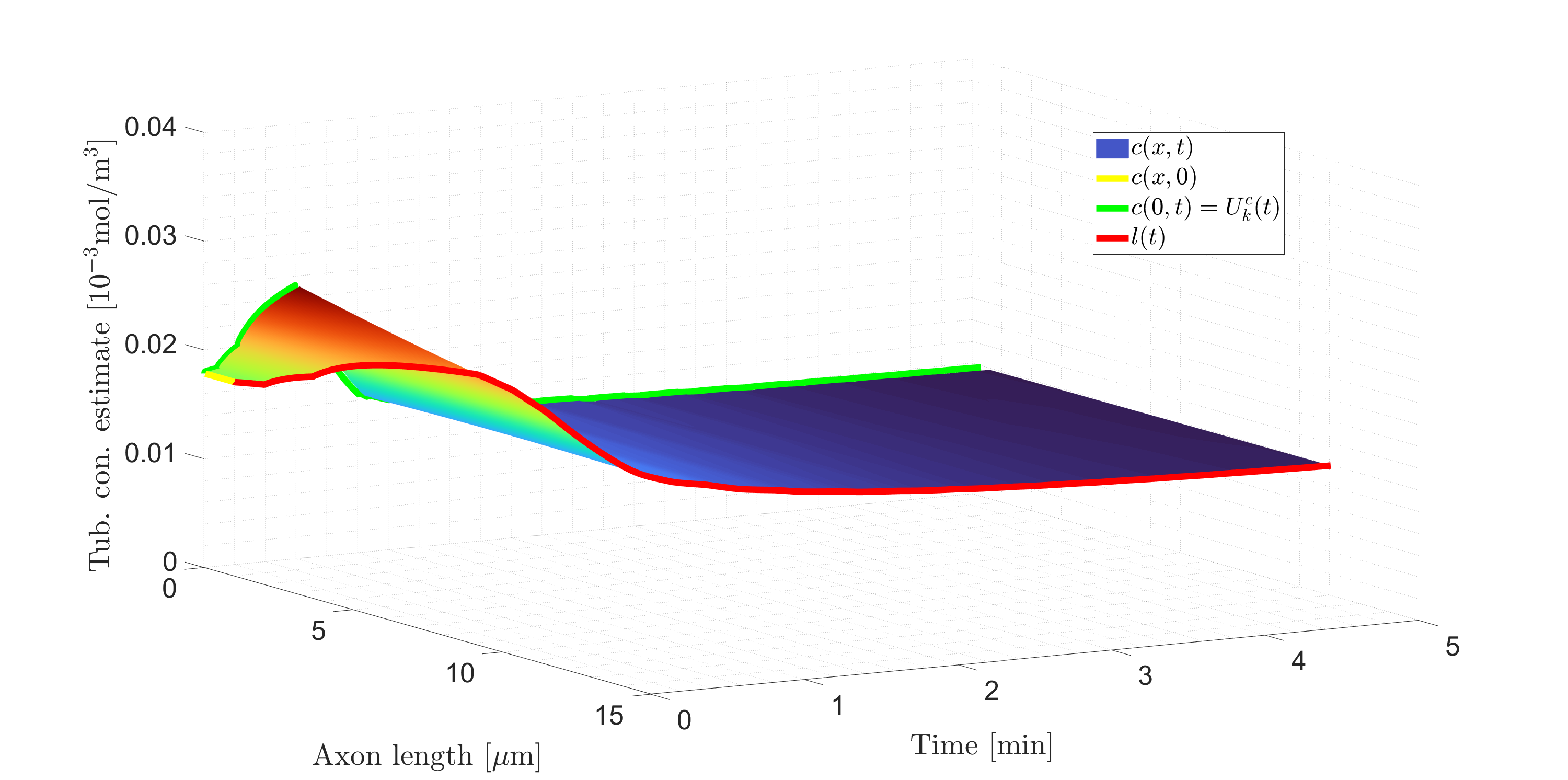}\label{fig:h1}}
\caption{The tubulin concentration governed by \eqref{sys1}-\eqref{sys5}, $c(x,t)$, converges to the steady-state tubulin concentration, $c_{\rm eq}(t)$ by about $t=4.5$min for both continuous control input, CETC and PETC. The axon length, $l(t)$, also converges to the desired axon length, $l_s$, by about $t=4$min for all sampling mechanisms.}
\label{fig:results}
\end{center}
\end{figure*}

\vspace{-0.25em}

\section{Conclusion}

This paper proposes a periodic-event triggering control of the axonal growth problem which is modeled as coupled PDE and nonlinear ODE. The nature of the actuation mechanism acting on the soma motivates the development of proof-based sampled-and-hold control techniques for practicality. By virtue of emulating the continuous-time feedback law and using recent techniques to carefully refine a specific category of continuous-time dynamic event trigger, we conceived a strategy that only requires periodic monitoring of the triggering condition and aperiodic updates of the control action. Given the vast number of neuron cells each with distinct parameters, future research will consider the unknown model parameters by prioritizing the estimation of these parameters and implementing adaptive control for this model. More precisely, exploiting Batch Least Squares Identifiers (BaLSI) techniques finite-time identification of the unknown parameters, and local exponential convergence are challenging but promising.

\vspace{-0.25em}

\section*{Acknowledgemts}

The authors express appreciation to Bhathiya Rathnayake for engaging in generous and selfless discussions on event-triggered control.

\bibliographystyle{IEEEtranS}
\bibliography{main.bib}
\end{document}